\documentclass[secthm,seceqn,amsthm,ussrhead,12pt]{amsart}
\usepackage[utf8]{inputenc}
\usepackage[english]{babel}

\usepackage{amssymb,amsmath,amsthm,amsfonts,xcolor,enumerate,hyperref,comment,longtable,cleveref}
\usepackage{soul}
\usepackage{times}
\usepackage{cite}
\usepackage{pdflscape}
\usepackage{ulem}
\usepackage[mathcal]{euscript}
\usepackage{tikz}
\usepackage{hyperref}
\usepackage{cancel}
\usepackage{stmaryrd}

\usetikzlibrary{arrows}

\newtheorem{teo}{Theorem}

\newtheorem{lemma}[teo]{Lemma}
\newtheorem{proposition}[teo]{Proposition}

\newtheorem{de}[teo]{Definition}

\usepackage{stmaryrd}
\usepackage{xcolor}

\begin{document}

\noindent{\Large Abelian groups gradings on null-filiform and one-parametric filiform Leibniz algebras}
\footnote{
The work was supported by the PCI of the UCA `Teor\'\i a de Lie y Teor\'\i a de Espacios de Banach',
 the PAI with project numbers FQM298, FQM7156,
Ministerio de Econom\'ia y Competitividad (Spain), grant MTM2016-79661-P (AEI/FEDER, UE, support included)
and  MTM2013-41208P, 
the 2014-2020 ERDF Operational Programme and by the Department of Economy, Knowledge, Business and University of the Regional Government of Andalusia FEDER-UCA18-107643,
 RFBR 20-01-00030,  CNPq    	302980/2019-9.}

   \medskip

 \medskip
   {\bf
   Antonio Jesús Calder\'{o}n$^{a}$, Luisa María Camacho$^{b}$, Ivan Kaygorodov$^{c,d}$ $\&$
    Bakhrom Omirov$^{e}$

    \medskip
}

\medskip

{\tiny

$^{a}$ Dpto. Matem\'{a}ticas. Universidad de C\'{a}diz. 11510 Puerto Real, C\'{a}diz, Spain

$^{b}$ Dpto. Matem\'{a}tica Aplicada I. Universidad de Sevilla. Avda. Reina Mercedes, s/n. 41012 Sevilla, Spain

$^{c}$ CMCC, Universidade Federal do ABC, Santo Andr\'e - SP, Brazil

$^{d}$ Moscow Center for Fundamental and Applied Mathematics, Moscow, GSP-1, 119991, Russia

$^{e}$  National University of Uzbekistan, Institute of Mathematics Uzbekistan Academy of Sciences, Tashkent 100174, Uzbekistan

\medskip

   E-mail addresses:
   \smallskip

  Antonio Jesús Calder\'{o}n (ajesus.calderon@uca.es)

\smallskip
 Luisa María Camacho (lcamacho@us.es)

\smallskip
    Ivan Kaygorodov (kaygorodov.ivan@gmail.com)

\smallskip

Bakhrom Omirov (omirovb@mail.ru)

}

\medskip

\noindent{\bf Abstract:} We classify, up to equivalences,  all abelian groups gradings    on null-filiform and one-parametric filiform Leibniz algebras. Any grading on a null-filiform Leibniz algebra is toral but there are non-toral gradings on  one-parametric filiform Leibniz algebras.

\medskip



\section{Introduction}

Gradings by abelian groups have played a key role in the study of Lie algebras and superalgebras, starting with the root space decomposition of the semisimple Lie algebras over the complex field, which is an essential ingredient in the Killing-Cartan classification of these algebras.
Gradings by a cyclic group appear in the connection between Jordan algebras and Lie algebras through the Tits-Kantor-Koecher construction, and in the theory of Kac-Moody Lie algebras.
Gradings by the integers or the integers modulo 2 are ubiquitous in Geometry.
Also there are some papers about non-group gradings \cite{c18,zus}.

In 1989, Patera and Zassenhaus \cite{pz89} began a systematic study of gradings by abelian groups on Lie algebras. They raised the problem of classifying the fine gradings, up to equivalence, on the simple Lie algebras over the complex
numbers. This problem has been settled now thanks to the work of many colleagues.
After that,  gradings of simple alternative and simple Malcev algebras \cite{elduque98},
the simple Kac Jordan superalgebra \cite{calderon10},
countless simple Lie algebras \cite{calderon14, draper12, draper14, draper06, draper09, draper10, draper16, elduque09, elduque10, elduque12, elduque13, elduque15} and
filiform Lie algebras \cite{bgr16} were described.

The concept of length of a Lie algebra was introduced by G\'{o}mez, Jim\'{e}nez-Merch\'{a}n and Reyes
in \cite{Gomez1, Gomez2}.  They distinguished an interesting family: algebras admitting a grading with
the greatest possible number of non-zero subspaces. Actually, the gradings with a large number of
non-zero subspaces enable us to describe the multiplication on the algebra more exactly.

In the past years, Leibniz algebras have been under active research (see, for example,
\cite{adashev17,Barnes, Nilradical, leib2,ikv,kppv,  Khudoyberdiyev13, Khudoyberdiyev14}). The main result on the structure of finite-dimensional Leibniz algebras asserts that a Leibniz algebra decomposes into a semidirect sum of a solvable radical and a semisimple Lie algebra \cite{Barnes}. Therefore, the main problem of the description of finite-dimensional Leibniz algebras consists of the study of solvable Leibniz algebras. Similarly to the case of Lie algebras the study of solvable Leibniz algebras is reduced to nilpotent ones \cite{Nilradical}.

Since the description of all $n$-dimensional nilpotent Leibniz algebras is an unsolvable task (even in the case of Lie algebras), we have to study nilpotent Leibniz algebras under certain conditions (conditions on index of nilpotency, various types of grading, characteristic sequence etc.) \cite{adashev17, Khudoyberdiyev13, Khudoyberdiyev14}. The well-known natural grading of nilpotent Lie and Leibniz algebras is very helpful when investigating of the properties of those algebras without restrictions on the grading. Indeed, we can always choose a homogeneous basis and thus the grading allows to obtain more explicit conditions for the structural constants. Moreover, such grading is useful for the investigation of cohomologies for the considered algebras, because it induces the corresponding grading of the group of cohomologies. Thus, it is very crucial to know what kind of grading admits a nilpotent Leibniz algebra.

In the present  paper we begin the study of gradings on Leibniz algebras by  classifying, up to equivalence,  of all abelian groups gradings of null-filiform and one-parametric filiform Leibniz algebras.

\section{Preliminaries}
In this section we give necessary definitions and preliminary results.

\begin{de} A  vector space with a bilinear multiplication $(L, \cdot)$ over the complex field   is called a Leibniz algebra if for any $x,y,z\in L$ the so-called Leibniz identity $$ x(yz) =(xy)z  - (xz)y$$  holds.
\end{de}

%From the Leibniz identity we conclude that the right annihilator $Ann_r(L) =\{x \in L \mid [y,x] = 0,\ \text{for \ all}\ y \in L \}$ of the Leibniz algebra $L$ is a two-sided ideal of $L$. Moreover, for a given element $x$ of a Leibniz algebra $L$, the right multiplication operators $\mathcal{R}_x \colon L \to L, \mathcal{R}_x(y)=[y,x],  y \in L$, are derivations.

For a given Leibniz algebra $L$ the sequence of two-sided ideals defined recursively as follows:
$$L^1=L, \ L^{k+1}=L^kL,  \ k \geq 1,$$
is said to be {\it the lower central series of $L$}.

\begin{de} A Leibniz algebra $L$ is said to be
nilpotent, if there exists $n\in\mathbb N$  such that $L^{n}=0$.
The minimal number $n$ with such property is said to be the index of nilpotency of the algebra $L$.
\end{de}

%Evidently, the index of nilpotency of an $n$-dimensional nilpotent algebra is not greater than $n+1$.

\begin{de} An $n$-dimensional Leibniz algebra $L$ is said to be null-filiform if $\dim L^i=n+1-i, \ 1\leq i \leq n+1$.
\end{de}

\begin{teo}[\cite{ayupov01}]   An arbitrary complex $n$-dimensional null-filiform non-Lie Leibniz algebra is isomorphic to the algebra
\[NF_n: \quad e_ie_1=e_{i+1}, \quad 1 \leq i \leq n-1,\]
where $\{e_1, e_2, \dots, e_n\}$ is a basis of the algebra $NF_n$.
\end{teo}

Actually, a nilpotent Leibniz algebra is null-filiform if and only if it is one-generated algebra. Notice that this notion has no sense in Lie algebra case, because they are at least two-generated.

\begin{de} An $n$-dimensional Leibniz algebra $L$ is said to be filiform if $\dim L^i=n-i$, for $2\leq i \leq n$.
\end{de}

\begin{de}
Let $G$ be a group.
An algebra $L$ is a $G$-graded algebra if and only if
 the vector space $L$ has the following decomposition $L=\bigoplus\limits_{g\in G} L_g$
and the multiplication law of $L$ has the following property $L_g  L_h \subset L_{g+h}, \forall g,h \in G.$  Any $L_g$ is called a homogeneous subspace and
the set of all $g$ such that $L_g\ne 0$ is  called the {support} of the grading.
\end{de}
For a $G$-graded algebra $L=\bigoplus\limits_{g\in G} L_g$
we will use the following notation $L_g:=\langle e_{i_1}, \ldots, e_{i_k} \rangle_g,$
if the homogeneous  subspace $L_g$ is generated by $ e_{i_1}, \ldots, e_{i_k}.$

\begin{de}
 Given
two groups  gradings $L=\bigoplus\limits_{g\in G} L_g$ and $L=\bigoplus\limits_{h\in H} L_h$ we
shall say that they are  {
equivalent} if there are:
\begin{enumerate}
\item a bijection $\sigma\colon G\to G'$
between the supports of the first and second gradings
respectively,
\item an algebra automorphism $\varphi$ of
$L$ such that $\varphi(L_g)=L_{\sigma(g)}$ for any $g\in G$.
\end{enumerate}
\end{de}

%Now let us define a natural graduation for a filiform Leibniz algebra.
%Given a filiform Leibniz algebra $L$, put $L_i=L^i/L^{i+1}, \ 1 \leq i\leq n-1$, and $gr(L) = L_1 \oplus
%L_2\oplus\dots \oplus L_{n-1}$. Then $[L_i,L_j]\subseteq L_{i+j}$ and we obtain the graded algebra $gr(L)$. If $gr(L)$ and $L$ are isomorphic, then we say that an algebra $L$ is naturally graded.

%Thanks to \cite{Ver} it is well known that there are two types of naturally graded filiform Lie algebras.
%In fact, the second type will appear only in the case when the dimension of the algebra is even.

%\begin{teo}[\cite{Ver}] \label{thm2.8} Any complex naturally graded filiform Lie algebra is isomorphic to one of the following non isomorphic algebras:
%\[L_{n}: [e_i,e_1]=-[e_1,e_i]=e_{i+1}, \quad 2\leq i \leq n-1.\]
%\[Q_{n} (n-even):\left\{\begin{aligned}
%{}[e_i,e_1] & =  -[e_1,e_i]=e_{i+1},&& 2\leq i \leq n-1,\\
%[e_i,e_{n-i}]  & =  -[e_{n-i},e_i]=(-1)^{i+1}\,e_{n},&& 2\leq i \leq %n-1.
%\end{aligned}\right.\]
%\end{teo}

Following \cite{OmiRak} we introduce the next class of filiform non-Lie Leibniz algebras.

\begin{de} \label{thm28} A complex $n$-dimensional  algebra admitting a basis $\{e_1,e_2,\dots,e_n\}$  such that the table of multiplication of the algebra has the  following form:
$$\begin{array}{llll}
e_ie_1 & =  & e_{i+1}, & \  2\leq i \leq {n-1},\\{}
e_1e_2 & =  &\theta e_n, & \\
%[e_j,e_2] & =  &\alpha_4e_{j+2}+\alpha_5 e_{j+3}+\dots+\alpha_{n+2-j}e_{n}, & 2\leq j\leq n-2.
\end{array} $$
with $\theta \in \mathbb{C}$ is called a {\it one-parametric filiform Leibniz algebra}.
\end{de}

%\bigskip

In the description of all abelian groups grading of our algebras,
we are using the techniques explained in the recent monograph \cite[Section 1.4]{elduque13}, which have been successfully used in the classification of abelian groups gradings on different classes of algebras (see for instance \cite{bgr16, calderon10, draper06, draper09}). To do that, { we will always  suppose that the grading group is generated by the support of the grading}. Roughly speaking, it is shown in the above references that  any group  grading, (with the group   finitely generated abelian),  is induced by a finitely generated abelian subgroup of
diagonalizable automorphisms of the automorphism group of the
algebra, (finite dimensional and over an algebraically closed field of characteristic zero),  under study. The homogeneous components are the simultaneous
eigenspaces relative to the given subgroup of automorphisms. Furthermore, up to equivalences of gradings,
any such subgroup is contained in the normalizer of some fixed
maximal torus of the automorphism group of the algebra.

%\medskip

A special kind of gradings arises when we consider the inducing
automorphisms not only in the normalizer of a maximal torus, but
in the torus itself.

\begin{de}\rm
A grading of an algebra is  said to be   {\it toral} if it is
produced by   automorphisms within  a torus of the automorphism
group of the algebra.
\end{de}

\section{Gradings on null-filiform Leibniz algebras}

By the above, in the first we are calculating the group of automorphisms of our algebra.
In the second, we are proving that the normalizer of a maximal torus in the group of automorphisms  is the same torus and so all of our gradings will be toral ones.
Finally, we are constructing all toral gradings on our algebra.

Let $A:=NF_n$ be a null-filiform Leibniz algebra.

\subsection{Automorphisms of  $A$}\label{autnf}
Let $f \in  {\rm Aut}(A)$, then
$
f(e_1)=a e_1+\displaystyle\sum_{j=2}^{n}a_j e_j,$
and $$f(e_i)=f(e_{i-1}e_1)=\ldots=
f((\ldots ((e_{1}e_1) e_1)\ldots )e_1)=
(\ldots ((f(e_{1})f(e_1)) f(e_1))\ldots) f(e_1),$$
hence
$\begin{array}{ll}
f(e_i)=a^{i} e_i+ a^{i-1}\displaystyle\sum_{j=1}^{n-i} a_{j+1}e_{j+i},&
1\leq i\leq n,
\end{array}$
with $a\neq 0.$ It is easy to see that all automorphisms of $A$ have the same form.

\subsection{Maximal Torus}\
It is easy to verify that a maximal torus is formed by:
$$\mathcal{T}=\left\{t_{a}:=\left(\begin{array}{llll}
a&0&\dots&0\\
0&a^2&\dots&0\\
\vdots&\vdots&\ddots&\vdots\\
0&0&\dots&a^n

\end{array}\right): a\in \mathbb{K}^{*}\right\}\cong \mathbb{K}^{*}.$$

\begin{lemma}\label{lemmatous}
	Let $\mathcal{N}(\mathcal{T})$ be the normalizer of $\mathcal{T}$. Then,  $\mathcal{N}(\mathcal{T})=\mathcal{T}.$
\end{lemma}

\begin{proof}
We have $\mathcal{N}(\mathcal{T})=\{M\in {\rm Aut}(A):\ MTM^{-1}\in \mathcal{T},\ \forall \ T\in \mathcal{T} \}.$

For $$M=\left(\begin{array}{lllll}
a&0&\dots&0&0\\
a_2&a^2&\dots&0&0\\
\vdots&\vdots&\ddots&\vdots&\vdots\\
a_{n-1}& aa_{n-2}& \dots&a^{n-1}&0\\
a_n& a a_{n-1}&\dots& a^{n-1} a_2&a^n
\end{array}\right) \mbox{and} \
T=\left(\begin{array}{lllll}
\lambda&0&\dots&0&0\\
0&\lambda^2&\dots&0&0\\
\vdots&\vdots&\ddots&\vdots&\vdots\\
0&0&\dots&\lambda^{n-1}&0\\
0&0&\dots&0&\lambda^n
\end{array}\right),$$
we have
$$MT=\left(\begin{array}{lllll}
a\lambda&0&\dots&0&0\\
a_2\lambda&a^2\lambda^2&\dots&0&0\\
\vdots&\vdots&\ddots&\vdots&\vdots\\
a_{n-1}\lambda& aa_{n-2}\lambda^2& \dots&a^{n-1}\lambda^{n-1}&0\\
a_n\lambda& a a_{n-1}\lambda^2&\dots& a^{n-1} a_2\lambda^{n-1}&a^n\lambda^n
\end{array}\right).$$

We need to prove the existence of $T'\in \mathcal{T}$ such that $MT=T'M.$
For $$T'=\left(\begin{array}{llll}
d&0&\dots&0\\
0&d^2&\dots&0\\
\vdots&\vdots&\ddots&\vdots \\
0&0&\dots&d^n
\end{array}\right)$$
we conclude that $d=\lambda$ and by choosing $\lambda\neq 1,0$ we get $a_2=a_3=\cdots=a_{n}=0$ and $M\in \mathcal{T}.$
\end{proof}

\subsection{Cyclic gradings}
A  cyclic grading  is a toral grading produced by a single
toral element $t_{a}$. In this case the grading is always
equivalent to a grading by a cyclic group. In order to study the grading induced by
$t_{a}$ on $A$, which is  the decomposition of $A$ as a
direct sum of eigenspaces of such toral element $t_{a}$, we are going to distinguish
different possibilities motivated by  the cardinal of the possible values of the set of eigenvalues of  $t_{a}$.

%\medskip

We distinguish the following cases:
\begin{itemize}
	\item $a^{i}\neq 1$ for $i=1,2,\dots,n.$ In this case, we have the next grading:
	$$A=\langle e_1\rangle_1\oplus\langle e_2\rangle_2\oplus\cdots\oplus \langle e_n\rangle_n:\ \mbox{ $\mathbb{Z}$-grading}.$$

	%\item $a=-1:$
	%\begin{itemize}
	%	\item If $n$ even, then we obtain
	%	$$A=\langle e_1,e_3,\dots,e_{n-1}\rangle_1\oplus\langle e_2,e_4,\dots,e_n\rangle_0: \ \mbox{ $\mathbb{Z}_2$-grading}.$$
	%	\item If $n$ odd, then we obtain
	%	$$A=\langle e_1,e_3,\dots,e_{n}\rangle_1\oplus\langle e_2,e_4,\dots,e_{n-1}\rangle_0: \ \mbox{ $\mathbb{Z}_2$-grading}.$$
%	\end{itemize}
%\item $a=\xi$ with $\xi^3=1$ and $\xi\neq \pm 1.$
%\begin{itemize}
%	\item If $n=3m,$ we have a $\mathbb{Z}_3$-grading:
%	$$A=\langle e_1,e_4,e_7,\dots,e_{n-2}\rangle_1\oplus\langle e_2,e_5,\dots,e_{n-1}\rangle_2\oplus\langle e_3,e_6,\dots,e_n\rangle_0.$$

%	\item If $n=3m-1,$ we have a $\mathbb{Z}_3$-grading:
%	$$A=\langle e_1,e_4,\dots,e_{n-1}\rangle_1\oplus\langle e_2,e_5,\dots,e_{n}\rangle_2\oplus\langle e_3,e_6,\dots,e_{n-2}\rangle_0.$$

%	\item If $n=3m-2,$ we have a $\mathbb{Z}_3$-grading:
%	$$A=\langle e_1,e_4,\dots,e_{n}\rangle_1\oplus\langle e_2,e_5,\dots,e_{n-2}\rangle_2\oplus\langle e_3,e_6,\dots,e_{n-1}\rangle_0.$$
%\end{itemize}

\item ${a}^i=1$ where  $a$ is an $i$-primitive root of 1, and $1 \leq i \leq n-1$, $i \in {\mathbb N}$.
If $i=1,$ we obtain the trivial grading. Thus, we can consider $2\leq i\leq n-1.$
Let us write  $n=mi+p$  with $0\leq p\leq i-1,$ $p \in {\mathbb N}$.  We have a $\mathbb{Z}_i$-grading
with the following homogeneous subspaces:

\begin{eqnarray}
\begin{array}{lll}\label{grad}
A_{\overline{0}}&=&\langle e_{i},e_{2i},\dots,e_{mi}\rangle\\
A_{\overline{1}}&=&\langle e_1,e_{i+1},\dots,e_{mi+1}\rangle\\
A_{\overline{2}}&=&\langle e_2,e_{i+2},\dots,e_{mi+2}\rangle\\
& &\qquad \qquad \ldots\\
A_{\overline{p}}&=&\langle e_p,e_{i+p},\dots,e_{mi+p}\rangle\\
A_{\overline{p+1}}&=&\langle e_{p+1},e_{i+p+1},\dots,e_{(m-1)i+p+1}\rangle\\
& &\qquad \qquad \ldots\\
A_{\overline{i-1}}&=&\langle e_{i-1},e_{2i-1},\dots,e_{(m-1)i+i-1}\rangle.
\end{array}
\end{eqnarray}

\end{itemize}

\begin{lemma}\label{gradings}
Let $A$ be a null-filiform Leibniz algebra of dimension $n$. Then any cyclic  grading is equivalent to only one of the following $n$ gradings:
\begin{enumerate}
	\item[{\rm (I)}] The trivial grading given by $A=\langle e_1,e_2,\dots,e_n\rangle.$
	\item[{\rm (II)}] The $\mathbb{Z}$-grading given by $A=\langle e_1\rangle_1\oplus\langle e_2\rangle_2\oplus\dots \oplus \langle e_n\rangle_n.$
	\item[{\rm (III)}] For any $2 \leq i \leq n-1$, the $\mathbb{Z}_i$-grading  given by
$A=A_{\overline{0}}\oplus A_{\overline{1}}\oplus \cdots\oplus A_{\overline{i-1}}$, where homogeneous  subspaces are described in grading {\rm (\ref{grad})}, being $n=mi+p$  with $0\leq p\leq i-1$.
\end{enumerate}
\end{lemma}
\begin{proof}
We have only to show that two different gradings in the lemma are not equivalent, but this is an immediate consequence of the fact that the cardinal of the support of any  grading is different to the cardinal of the support of any other different grading.
\end{proof}

\subsection{Classification theorem}
\begin{lemma}\label{morse1}
	If $e_1$ is a homogeneous element of an abelian  group grading (that is, $e_1\in A_x$ for some $x\in G$), then the grading is one of the list of Lemma \ref{gradings}.
\end{lemma}

\begin{proof}
	Let $e_1\in A_x$ and $x\in G$ be. Let $i$ be the order of $x$.
	\begin{itemize}
		\item $i>n.$ Let $j\leq n$ be, we have
		$e_j= (\ldots ((e_1\underbrace{e_1)e_1)\dots )e_1}_{j-1}\in A_{jx}$
		and the $\mathbb{Z}$-grading
		$$A=\langle e_1\rangle\oplus \langle e_2\rangle\oplus \cdots \oplus \langle e_n\rangle.$$
		
		\item $i\leq n.$ Then, $e_1\in A_x,\ e_2\in A_{2x},\dots, e_{i-1}\in A_{(i-1)x}$ with $|\{x,2x,\dots,(i-1)x\}|=i-1.$
		Let $j\geq i$ be, $j=im+p$ with $0\leq p<i.$ Then
		$e_j= (\ldots ((e_1\underbrace{e_1)e_1)\dots)e_1}_{j-1}\in A_{im+p}=A_{\overline{p}}.$ Thus,

		$$A=\langle e_1,e_{1+i},\dots\rangle_{\overline{1}}\oplus \langle e_2,e_{2+i},\dots\rangle_{\overline{2}}
		\oplus\cdots\oplus \langle e_{i-1},e_{2i-1},\dots\rangle_{\overline{i-1}}\oplus \langle e_i,e_{2i},\dots\rangle_{\overline{0}}$$
	\end{itemize}
	\end{proof}

Lemma \ref{lemmatous} and Lemma \ref{morse1} will allow us to assert that all of the abelian group gradings of $A$ are the ones given in Lemma \ref{gradings}. Indeed, since by Lemma \ref{lemmatous}, $\mathcal{N}(\mathcal{T})=\mathcal{T}, $ we know that all of the gradings are toral. Hence any grading is  the simultaneous
eigenspaces decomposition of $A$ respect to a given (abelian) subgroup of semisimple automorphisms $\mathcal{U}$ contained in $\mathcal{T}$. Now. since $$\mathcal{T}=\{t_{a} : a \in {\mathbb K}^*\},$$ (see 3.1.2), and any $t_{a}$ is of the form  $t_{a}:=\left(\begin{array}{llll}
a&0&\dots&0\\
0&a^2&\dots&0\\
\vdots&\vdots&\ddots&\vdots\\
0&0&\dots&a^n
\end{array}\right)$ respect to the basis $\{e_1,e_2,...,e_n\}$ we have that $e_1$ belongs to the eigenspace associated to the eigenvalue $a$, for any $t_{a}  \in \mathcal{T}$ and so $e_1$ belongs to some  simultaneous
eigenspace respect to the  decomposition of $A$ relative to the simultaneous
eigenspaces decomposition of $A$ through the abelian family of semisimple automorphisms $\mathcal{U}$. From here  $e_1$ is a homogeneous element of any abelian  group grading of $A$ and we can apply Lemma \ref{morse1} to conclude the next  theorem:

\begin{teo}\label{teo1}
Any abelian group grading of a null-filiform Leibniz algebra is equivalent to only one of the list of Lemma \ref{gradings}.
\end{teo}

%{\color{red}\begin{remark}\rm
%We note that we could have omitted the computation of the cyclic gradings of $F_1(\alpha_4,\alpha_5,\dots,\alpha_n,\theta)$, which has been carried  out  in subsection 3.1.3. Indeed, we could have started our study with Lemma \ref{morse1}, followed by Lemma \ref{lemmatous}, then  by the proof of Lemma \ref{gradings} and finally by Theorem \ref{teo1}. However, we opted for the present development of our study in order to follow a classical approach to the group gradings of an algebra (see for instance \cite{calderon10, draper09}). This note also applies to the  classifications of the abelian group gradings of filiform Leibniz algebras in subsection 3.2.
%\end{remark}}

\section{Gradings on one-parametric filiform Leibniz algebras}

We will follow the same program than in the previous section to classify  the abelian group gradings on  one-parametric filiform Leibniz algebras. First we compute   the group of automorphisms of our algebras. Then, in order to  find a maximal torus and
compute its normalizer, we will have two distinguish two cases attending to the possible nullity of the parameter. We will obtain that just in  case this parameter is zero  all of the gradings are   toral ones. In this case we will also have that   any abelian group grading is necessarily cyclic. For the remaining  cases  we will have to develop new tools for their study.

%\medskip

\subsection{Automorphisms of $F_1$.}
Our first goal is to compute its group of automorphisms.
Let ${\rm dim}(F_1) \geq 4$ and $f\in {\rm Aut}(F_1)$. Then
\begin{longtable}{rclcrcl}
$f(e_1)$ & $=$ & $\displaystyle\sum_{k=1}^{n} a_k e_k$  \mbox{ and }
$f(e_2)$ & $=$ & $\displaystyle\sum_{k=1}^{n} b_k e_k.$
\end{longtable}

It is easy to see that $a_1b_2\neq 0.$ From
$f(e_1)^2=f(e_1^2)=0$ we find $f(e_1)=a_1 e_1+a_ne_n;$
and from  $f(e_2)^2=f(e_2^2)=0$ we have $b_1=0$.
Also from  $f(e_1)f(e_2)=\theta f(e_n)$ we obtain $\theta (1-a_1^{n-3})=0.$

If we rename $a_1=a$ and $b_2=b,$
we have
\begin{longtable}{rcl}
$f(e_1)$ & $=$ & $a e_1+ a_n e_n,$\\
$f(e_i)$ & $=$ & $a^{i-2}(b e_i+b_3 e_{i+1}+\cdots +b_{n-i+2}e_n), \ 2\leq i\leq n,\quad a b\neq 0,$\\
$\theta(1-a^{n-3} )$ & $=$ & $0$.
\end{longtable}

So we get:
\begin{equation}\label{June1}
{{\rm Aut}}(F_1)=
\left\{\left(\begin{array}{lllll}
a&0&0&\dots&0\\
0&b&0&\dots&0\\
0&b_3&ab&\dots&0\\
0&b_4&ab_3&\dots&0\\
\vdots&\vdots&\vdots&\ddots&\vdots\\
a_n&b_n&ab_{n-1}&\dots&a^{n-2}b

\end{array}\right):
\begin{array}{l}
a_n, b_3,...,b_n \in \mathbb{C} \hspace{0,1cm} {\rm and}  \hspace{0,1cm} a,b\in \mathbb{C}^{*}, \\ \hspace{0,1cm} {\rm  } \hspace{0,1cm}
 \theta(1-a^{n-3})=0
 \end{array}\right\}.
 \end{equation}

 To  calculate a maximal torus in the above   group of automorphisms   we are going to distinguish two cases.

\subsection{Case in which $\theta=0$}

%\medskip

Since in this case  $${{\rm Aut}}(F_1)=\left\{\left(\begin{array}{lllll}
a&0&0&\dots&0\\
0&b&0&\dots&0\\
0&b_3&ab&\dots&0\\
0&b_4&ab_3&\dots&0\\
\vdots&\vdots&\vdots&\ddots&\vdots\\
a_n&b_n&ab_{n-1}&\dots&a^{n-2}b
\end{array}\right):
a_n, b_3,...,b_n \in \mathbb{C}, \ a,b\in \mathbb{C}^{*}  \right\}$$

for $n \geq 3$ (the case $n=3$ is easy to verify), we get  that a maximal torus is formed by:
\begin{equation}\label{toruss}
\mathcal{T}=\left\{t_{a,b}:=\left(\begin{array}{ccccc}
a&0&0&\dots&0\\
0&b&0&\dots&0\\
0&0&ab&\dots&0\\
\vdots&\vdots&\vdots&\ddots&\vdots\\
0&0&0&\dots&a^{n-2}b

\end{array}\right):a,b\in \mathbb{C}^{*}\right\} \cong \mathbb{C}^{*}  \times \mathbb{C}^{*}.
\end{equation}

By  similar calculations to Lemma \ref{lemmatous} one can prove the following result:

\begin{lemma}\label{yes1}
	Let $\mathcal{N}(\mathcal{T})$ be the normalizer of $\mathcal{T}$. Then,  $\mathcal{N}(\mathcal{T})=\mathcal{T}.$
\end{lemma}

Denote by  $A$ a one-parametric filiform Leibniz algebra  with  $\theta=0$.  Let us compute the grading on $A$ induced by only an element $t_{a,b}$ in our maximal torus.
We will denote $d_1=a,$ $d_2=b$ and $d_i=a^{i-2}b$ with $3\leq i\leq n$ the diagonal of the matrix $t_{a,b}$ respect our fixed basis.

\medskip

We distinguish the following cases:
\begin{enumerate}
	\item $a=b.$
	\begin{enumerate}
		\item[$(1.1)$] If $a=b=1.$ In this case, we have the trivial grading:
		$$A=\langle e_1,e_2,\dots,e_n\rangle.$$

		\item[$(1.2)$] If $a=b$ and $a^i=1$,  being   $a$  an $i$-primitive root of 1  and $2 \leq i \leq n-1$, $i \in {\mathbb N}.$
 Let $n=im+p$ be with $0\leq p<i.$ Then, we get the $\mathbb{Z}_{i}$-grading
		
\begin{longtable}{lllllllllll}

		$A$ &$=$& 		&&    $\langle e_{i+1},$ &$e_{2i+1}$  &$\dots,$&$e_{1+(m-1)i},$& $e_{1+mi}\rangle_{\overline{0}},$\\
		&&$\oplus$ & $\langle e_1, e_2,$ &$e_{2+i},$&$e_{2+2i},$&$\dots,$&$e_{2+(m-1)i},$ & $e_{2+mi}\rangle_{\overline{1}}$ \\

		&&  $\oplus$ & $\langle e_3,$ &$e_{3+i},$ &$e_{3+2i},$ &$\dots,$ &$e_{3+(m-1)i},$ & $e_{3+mi}\rangle_{\overline{2}}$& \\
		&&  &$\ldots$ & \\
		&&  $\oplus$ & $\langle e_{p},$ &$e_{p+i},$ &$e_{p+2i},$ &$\dots,$ &$e_{p+(m-1)i},$ &  $e_{p+mi}\rangle_{\overline{{p-1}}}$ \\
		&&  $\oplus$ & $\langle e_{p+1},$ &$e_{p+1+i},$ &$e_{p+1+2i},$ &$\dots,$ &$e_{p+1+(m-1)i}\rangle_{\overline{p}}$ \\
		&&  & $\ldots$ \\
		&&  $\oplus$ & $\langle e_{i},$ &$e_{2i},$ &$e_{3i},$ &$\dots,$& $e_{mi}\rangle_{\overline{i-1}}.$\\
		\end{longtable}

		\item[$(1.3)$] If $a=b$ and $a^i\neq 1, 0<i<n.$
		We have the $\mathbb{Z}$-grading
		$$A=\langle e_1,e_2\rangle_1\oplus \langle e_3\rangle_2\oplus \langle e_4\rangle_3\oplus\cdots \oplus \langle e_{n-1}\rangle_{n-2}\oplus \langle e_n\rangle_{n-1}.$$
		
	\end{enumerate}
\item $a\neq b.$ We have $d_1=a,$ $d_2=b,$ $d_i=a^{i-2}b$ with $3\leq i\leq n.$
\begin{enumerate}
\item[$(2.1)$] If $a=1.$ We obtain a $\mathbb{Z}_2$-grading
	$$A=\langle e_1\rangle_{\overline{0}} \oplus \langle e_2,e_3,\dots,e_n\rangle_{\overline{1}}.$$

\item[$(2.2)$] If $a=-1$ and $b=1.$

	We get the $\mathbb{Z}_2$-grading
	$$A= \langle e_2,e_4, e_6,\ldots \rangle_{\overline{0}} \oplus \langle e_1,e_3,e_5, e_7,\ldots \rangle_{\overline{1}}.$$

\item[$(2.3)$] If $a=-1$ and $b \neq 1.$

We get the $\mathbb{Z} \times  \mathbb{Z}_2  $-grading
	$$A=\langle e_1\rangle _{(0,{\overline{1}})}\oplus \langle e_2,e_4, e_6, \dots \rangle _{(1,{\overline{0}})}\oplus \langle e_3,e_5, e_7, \dots \rangle_{(1,{\overline{1}})}.$$
	
\item[$(2.4)$] If $a\notin \{1,-1\}.$

	\begin{enumerate}
		
	\item[$(2.4.1)$] If $b=1,$ then $d_1=a,$ $d_2=1,$ $d_i=a^{i-2}$ with $3\leq i\leq n.$
	We can distinguish two cases:
	\begin{enumerate}
		\item If there exists $i$ with $3\leq i\leq n-2$ such that $a^i=1.$ Let $n = mi+p, $ $0\leq p\leq i-1.$ We have the following $\mathbb{Z}_i$-grading

\begin{longtable}{lllllllllll}

		$A$ &$=$& & &  $\langle 		 e_2,$ & $e_{2+i},$ &$e_{2+2i},$ &$\ldots,$ &$e_{2+(m-1)i},$& $e_{2+mi}\rangle_{\overline{0}}$ \\
		
		&&  $\oplus$ & $\langle e_1,$ & \ $e_3,$ &$e_{3+i},$ &$e_{3+2i},$ &$\ldots,$ &$e_{3+(m-1)i},$ &$e_{3+mi}\rangle_{\overline{1}}$ \\
		&&  &$\ldots$ & \\
		&&  $\oplus$ && $\langle  e_{p},$ & $e_{p+i},$ & $e_{p+2i},$ &$\ldots,$ & $e_{p+(m-1)i},$ & $e_{p+mi}\rangle_{{\overline{p-2}}}$ \\
		&&  $\oplus$ && $\langle  e_{p+1},$ &$e_{p+1+i},$ &$e_{p+1+2i},$ &$\ldots,$ &$e_{p+1+(m-1)i}\rangle_{{\overline{p-1}}}$ \\
		&&  & $\ldots$ \\
		&&  $\oplus$ && $\langle  e_{i+1},$ &$e_{2i+1},$ &$e_{3i+1},$ &$\ldots,$& $e_{i+1+(m-1)i}\rangle_{{\overline{i-1}}}.$
		
		\end{longtable}

		\item If $a^i \neq 1$ for any $i,$ $3\leq i\leq n-2.$ We have the following $\mathbb{Z}$-grading:
		$$A=\langle e_2\rangle_{0} \oplus \langle e_1,e_3\rangle_{1} \oplus \langle e_4\rangle_2 \oplus\cdots \oplus \langle e_{n-1}\rangle_{n-3} \oplus \langle e_n\rangle_{n-2}.$$
		
	\end{enumerate}

\item[$(2.4.2)$] If $b\neq 1.$
\begin{enumerate}
	\item $d_i\neq d_j$ for all $i,j$ with $1\leq i,j\leq n.$ We have
	the following $\mathbb{Z}$-grading:
	$$A=\langle e_1\rangle_1\oplus \langle e_2\rangle_2\oplus \langle e_3\rangle_{3}\oplus \langle e_4\rangle_{4}\oplus\cdots \oplus \langle e_{n-1}\rangle_{n-1}\oplus \langle e_n\rangle_{n}.$$
	
	\item  there exist $k,l$ with $k\neq l$ such that $d_k=d_l$ with $3\leq k<l\leq n.$
	Thus, $a^{i}=1.$ Let $n= mi+p $ be. We have
	the following $\mathbb{Z} \times \mathbb{Z}_{i}$-grading:

\begin{longtable}{lllllllllll}
		$A$ &=&  & $\langle e_1  \rangle_{(0,{\overline{1}})}$\\
		&&  $\oplus$&$  \langle  e_2,$ & $e_{2+i},$ &$e_{2+2i},$&$\ldots,$&$e_{2+(m-1)i},$&$e_{2+mi}\rangle_{(1,{\overline{0}})}$ \\
		&&  $\oplus$&$  \langle   e_3,$ &$e_{3+i},$ &$e_{3+2i},$ &$\ldots,$ &$e_{3+(m-1)i},$ &$e_{3+mi}\rangle_{(1,{\overline{1}})}$ \\
		&&  &$\ldots$ & \\
		&&  $\oplus $&$ \langle  e_{p},$ &$e_{p+i},$ &$e_{p+2i},$ &$\ldots,$ &$e_{p+(m-1)i},$ &$e_{p+mi}\rangle_{(1,{\overline{p-2}})}$ \\
		&&  $\oplus$&$  \langle  e_{p+1},$ &$e_{p+1+i},$ &$e_{p+1+2i},$ &$\ldots,$ &$e_{p+1+(m-1)i}\rangle_{(1,{\overline{p-1}})}$ \\
		&&  & $\ldots$ \\
		&&  $\oplus$&$  \langle  e_{i+1},$ &$e_{2i+1},$ &$e_{3i+1},$ &$\ldots,$& $e_{i+1+(m-1)i}\rangle_{(1,{\overline{i-1}})}.$
		\end{longtable}

	\item  there exists $i,$ $3\leq i\leq n$ such that $d_1=d_i$ and $d_1\neq d_j, i\neq j.$
	Thus,
	$b=\frac{1}{a^{i-3}}.$ We put $a^{i-3}\neq 1$ because in other case $b=1$ and it gives the case (2.4.1.B).   We have
	the following $\mathbb{Z}$-grading:
	$$A=\langle e_{2}\rangle_{-i+3} \oplus \cdots \oplus  \langle e_{i-2}\rangle_{-1} \oplus \langle e_{i-1}\rangle_{{{0}}} \oplus \langle e_1,e_i\rangle_{1} \oplus  \langle e_{i+1}\rangle_{{{2}}} \oplus \cdots  \oplus \langle e_{n}\rangle_{{{n-i+1}}},$$

	\item  there exist $k,l$ with $k>l$ such that $d_1=d_k=d_l$ and $k-l=i.$
	Thus $b=a^{3-l}$ and $a^{i}=1.$ Let $n=mi+p$ with $0\leq p< i$ be.
	We have
	the following $\mathbb{Z}_{i}$-grading:

\begin{longtable}{lllllllllll}
		$A$   &$=$&             &                     &$\langle  e_{1},$ & $e_{i+1},$& $e_{2i+1},$ & $e_{3i+1},$  & $\ldots \rangle_{{\overline{1}}}$ \\
		    &&          &$\oplus$&    $\langle   e_2,$ &   $e_{i+2},$            & $e_{2i+2},$ & $e_{3i+2},$  & $\ldots  \rangle_{{\overline{2}}}$ \\
		    &&         & &$\ldots$& \\
		    &&             &$\oplus$                      &$\langle  e_{i}, $& $e_{2i},$ & $e_{3i},$  & $e_{4i},$ & $\ldots \rangle_{{\overline{0}}}.$ \\

		\end{longtable}

\end{enumerate}

\end{enumerate}

\end{enumerate}
\end{enumerate}

%\medskip

Since one-dimensional filiform Leibniz algebras    and two-dimensional  filiform Leibniz algebras with  $\theta=0$ have  zero product we consider  filiform Leibniz algebras, with  $\theta=0$, of dimension greater or equal than three.

\begin{lemma}\label{gradings2}
Let $A$ be a one-parametric filiform Leibniz algebra    with  $\theta=0$ and  of dimension $n \geq 3$.
Then any  cyclic  grading is  equivalent to only  one of the following:
\begin{enumerate}
	\item[{\rm (1)}] The trivial grading given by $A=\langle e_1,e_2,\dots,e_n\rangle.$

	\item[{\rm (2)}] The   $\mathbb{Z}$-grading  given by
	$$A=\langle e_1\rangle_1\oplus \langle e_2\rangle_2\oplus \langle e_3\rangle_{3}\oplus \langle e_4\rangle_{4}\oplus\cdots \oplus \langle e_{n-1}\rangle_{n-1}\oplus \langle e_n\rangle_{n}.$$

\item[{\rm (3)}] The $\mathbb{Z}$-grading given by
		$$A=\langle e_1,e_2\rangle_1\oplus \langle e_3\rangle_2\oplus \langle e_4\rangle_3\oplus\cdots \oplus \langle e_{n-1}\rangle_{n-2}\oplus \langle e_n\rangle_{n-1}.$$

	%\item The $\mathbb{Z}$-grading given by
	%	$$A=\langle e_2\rangle_{0} \oplus \langle e_1,e_3\rangle_{1} \oplus \langle e_4\rangle_2 \oplus\cdots \oplus \langle e_{n-1}\rangle_{n-3} \oplus \langle e_n\rangle_{n-2}.$$

\item[{\rm (4)}] For any $3 \leq i \leq n$, the  $\mathbb{Z}$-grading given by
	$$A=\langle e_{2}\rangle_{-i+3} \oplus \cdots \oplus \langle e_{i-2}\rangle_{-1} \oplus \langle e_{i-1}\rangle_{{{0}}} \oplus \langle e_1,e_i\rangle_{1} \oplus  \langle e_{i+1}\rangle_{{{2}}} \oplus  \cdots  \oplus \langle e_{n}\rangle_{{{n-i+1}}}.$$

\item[{\rm (5)}] The $\mathbb{Z}_2$-grading given by
	$$A=\langle e_1\rangle_{\overline{0}} \oplus \langle e_2,e_3,\dots,e_n\rangle_{\overline{1}}.$$

\item[{\rm (6)}]  For $n \geq 4$, the $\mathbb{Z}_2$-grading given by
	$$A= \langle e_2,e_4, e_6,\ldots \rangle_{\overline{0}} \oplus \langle e_1,e_3,e_5, e_7,\ldots \rangle_{\overline{1}}.$$

\item[{\rm (7)}] The $\mathbb{Z} \times  \mathbb{Z}_2  $-grading given by
	$$A=\langle e_1\rangle _{(0,{\overline{1}})}\oplus \langle e_2,e_4, e_6, \dots \rangle _{(1,{\overline{0}})}\oplus \langle e_3,e_5, e_7, \dots \rangle_{(1,{\overline{1}})}.$$

\item[{\rm (8)}] For any $2 \leq i \leq n-1$, $(i,n) \neq (2,3)$,  the $\mathbb{Z}_{i}$-grading given by
		\begin{longtable}{lllllllllll}
	$A$ &$=$& 		&&    $\langle e_{i+1},$ &$e_{2i+1},$  &$\dots,$&$e_{1+(m-1)i},$& $e_{1+mi}\rangle_{\overline{0}}$\\
		&&$\oplus$ & $\langle e_1, e_2,$ &$e_{2+i},$&$e_{2+2i},$&$\dots,$&$e_{2+(m-1)i},$ & $e_{2+mi}\rangle_{\overline{1}}$ \\

		&&  $\oplus$ & $\langle e_3,$ &$e_{3+i},$ &$e_{3+2i},$ &$\dots,$ &$e_{3+(m-1)i},$ & $e_{3+mi}\rangle_{\overline{2}}$& \\
		&&  &$\ldots$ & \\
		&&  $\oplus$ & $\langle e_{p},$ &$e_{p+i},$ &$e_{p+2i},$ &$\dots,$ &$e_{p+(m-1)i},$ &  $e_{p+mi}\rangle_{\overline{{p-1}}}$ \\
		&&  $\oplus$ & $\langle e_{p+1},$ &$e_{p+1+i},$ &$e_{p+1+2i},$ &$\dots,$ &$e_{p+1+(m-1)i}\rangle_{\overline{p}}$ \\
		&&  & $\ldots$ \\
		&&  $\oplus$ & $\langle e_{i},$ &$e_{2i},$ &$e_{3i},$ &$\dots,$& $e_{mi}\rangle_{\overline{i-1}}.$\\
		\end{longtable}

    \item[{\rm (9)}]   For any $3 \leq i \leq n-1$, $(i,n) \neq (3,4)$, the $\mathbb{Z}_i$-grading given by

\begin{longtable}{lllllllllll}

		$A$ &$=$& & &  $\langle 		 e_2,$ & $e_{2+i},$ &$e_{2+2i},$ &$\ldots,$ &$e_{2+(m-1)i},$& $e_{2+mi}\rangle_{\overline{0}}$ \\
		
		&&  $\oplus$ & $\langle e_1,$ & \ $e_3,$ &$e_{3+i},$ &$e_{3+2i},$ &$\ldots,$ &$e_{3+(m-1)i},$ &$e_{3+mi}\rangle_{\overline{1}}$ \\
		&&  &$\ldots$ & \\
		&&  $\oplus$ && $\langle  e_{p},$ & $e_{p+i},$ & $e_{p+2i},$ &$\ldots,$ & $e_{p+(m-1)i},$ & $e_{p+mi}\rangle_{{\overline{p-2}}}$ \\
		&&  $\oplus$ && $\langle  e_{p+1},$ &$e_{p+1+i},$ &$e_{p+1+2i},$ &$\ldots,$ &$e_{p+1+(m-1)i}\rangle_{{\overline{p-1}}}$ \\
		&&  & $\ldots$ \\
		&&  $\oplus$ && $\langle  e_{i+1},$ &$e_{2i+1},$ &$e_{3i+1},$ &$\ldots,$& $e_{i+1+(m-1)i}\rangle_{{\overline{i-1}}}.$
		
		\end{longtable}

\item[{\rm (10)}] For any $3 \leq i \leq n-2$, the $\mathbb{Z} \times \mathbb{Z}_{i}$-grading given by

\begin{longtable}{lllllllllll}
		$A$ &=&  & $\langle e_1  \rangle_{(0,{\overline{1}})}$\\
		&&  $\oplus$&$  \langle  e_2,$ & $e_{2+i},$ &$e_{2+2i},$&$\ldots,$&$e_{2+(m-1)i},$&$e_{2+mi}\rangle_{(1,{\overline{0}})}$ \\
		&&  $\oplus$&$  \langle   e_3,$ &$e_{3+i},$ &$e_{3+2i},$ &$\ldots,$ &$e_{3+(m-1)i},$ &$e_{3+mi}\rangle_{(1,{\overline{1}})}$ \\
		&&  &$\ldots$ & \\
		&&  $\oplus $&$ \langle  e_{p},$ &$e_{p+i},$ &$e_{p+2i},$ &$\ldots,$ &$e_{p+(m-1)i},$ &$e_{p+mi}\rangle_{(1,{\overline{p-2}})}$ \\
		&&  $\oplus$&$  \langle  e_{p+1},$ &$e_{p+1+i},$ &$e_{p+1+2i},$ &$\ldots,$ &$e_{p+1+(m-1)i}\rangle_{(1,{\overline{p-1}})}$ \\
		&&  & $\ldots$ \\
		&&  $\oplus$&$  \langle  e_{i+1},$ &$e_{2i+1},$ &$e_{3i+1},$ &$\ldots,$& $e_{i+1+(m-1)i}\rangle_{(1,{\overline{i-1}})}.$
		\end{longtable}

\item[{\rm (11)}] For any $3 \leq i \leq n-2$,  the $\mathbb{Z}_{i}$-grading given by

\begin{longtable}{lllllllllll}
		$A$   &$=$&             &                     &$\langle  e_{1},$ & $e_{i+1},$& $e_{2i+1},$ & $e_{3i+1},$  & $\ldots \rangle_{{\overline{1}}}$ \\
		    &&          &$\oplus$&    $\langle   e_2,$ & \ $e_{i+2},$            & $e_{2i+2},$ & $e_{3i+2},$  & $\ldots  \rangle_{{\overline{2}}}$ \\
		    &&         & &$\ldots$& \\
		    &&             &$\oplus$                      &$\langle  e_{i}, $& $e_{2i},$ & $e_{3i},$  & $e_{4i},$ & $\ldots \rangle_{{\overline{0}}}.$ 
		\end{longtable}

\end{enumerate}
\end{lemma}

\begin{proof}

By the above discussion, we just have to prove that two different gradings of the lemma are not equivalent.
 %Since two gradings with supports of different cardinal are not equivalent, we have that (1),  and (2) for $n \neq 3$  are not equivalent to any other grading in the list. The grading (3) for $n=3$..

Since two gradings with supports of different cardinal are not equivalent, we have that (1), (2) and (3) are not equivalent among them.

Consider now the gradings in the family, depending on $i$,  (4) having all of them a support with cardinal $n-1$. From here these are not equivalent to (1)  or (2). Clearly the grading (3) neither is  equivalent to any grading in (4), because in the opposite case there should be an automorphism $\phi$ of $A$ such that
$\phi(\langle e_1,e_2\rangle)=\langle e_1,e_i\rangle$ with $i \geq 3$, what would be a contradiction. If we fix now $ 3 \leq i,j \leq n$ with $i \neq j$ and consider the two gradings of type (4) associated to $i$ and $j$. In case both gradings were equivalent, we would have an automorphism $\phi$ of $A$ such that $\phi(\langle e_1,e_i\rangle)=\langle e_1,e_j\rangle$, but the dimension of the subalgebra of $A$ generated by $\{e_1,e_i\}$ is $n-i+2$ while the one generated by $\{e_1,e_j\}$ is $n-j+2$. Hence there not exists such $\phi$ because $i \neq j$. We conclude that two different gradings of the family (4) are not equivalent.

Consider now the grading (5). This is not equivalent to any of the previous ones because of the different cardinals of the supports, (in the case (3) when $n=3$ as consequence of  there is not any automorphism satisfying $\phi(\langle e_1,e_2\rangle)=\langle e_2,e_3\rangle$ while in the case (4) when $n=3$ because   there is not any automorphism satisfying $\phi(\langle e_1,e_3\rangle)=\langle e_2,e_3\rangle$).

As above, the grading (6) is not equivalent to any grading (1)-(5).

Respect to the grading (7), we get as above that this is not equivalent to (1)-(6).

Taking into account that  the dimension of the homogeneous subspace containing $e_1$ is at least three in grading (8), we have as above that grading (8) is not equivalent to (1)-(7).

We observe that any grading in (9) has at least two homogeneous components with dimensions greater or equal that two, so they are not equivalent to any grading (1)-(7).

Since in (8) any grading has $e_1$ and $e_2$ in the same  homogeneous component, we also have that (9) is not equivalent to (8).

Taking now into account that in gradings (10) there is always at least four nonzero homogeneous components, that $\langle e_1 \rangle$ is always a  homogeneous component and that there exists at least one homogeneous component with dimension greater o equal than two, we get that gradings in (10) are not equivalent to any grading in (1)-(9).

Finally, by observing that in (11) all of the gradings have at least two homogeneous components with dimensions greater or equal than two, that $e_1$ does not belong to a one-dimensional homogeneous component, that $e_2$ does not belong to the same component than $e_1$ and that $e_3$ belongs to the same component than $e_1$ just in case $i=2$, we can assert that any grading in (11) is not equivalent to any grading in (1)-(10).

\end{proof}

Next, we are going to show that any abelian  group grading on $A$ is equivalent to a cyclic grading.

\begin{lemma}\label{JC}
	If $e_1, e_2$ are homogeneous elements of an abelian  group grading of $A$ (that is, $e_1\in A_x, e_2 \in  A_y $ for some $x,y\in G$), then the grading of $A$ is one of the list in Lemma \ref{gradings2}.
\end{lemma}

\begin{proof}
	Let $e_1\in A_x, e_2 \in A_y$ and $x,y\in G$ be; and let denote by $i$ be the order of $x$ (denoted by ${\rm ord}(x)=i$). We are going to distinguish two cases.
	
	In the first one $x=y.$ That is $e_1,e_2 \in A_x$. We can consider the next possibilities:
	\begin{itemize}
		\item We have $i \geq n.$ We get by the grading $e_3 = e_2 e_1\in A_{2 x}$, $e_4 = e_3 e_1\in A_{3 x}$ and so $e_j = e_{j-1} e_1\in A_{({j-1}) x}$ for any  $j\leq n$.  From here we  have the grading
		$$A=\langle e_1,e_2\rangle_x\oplus \langle e_3\rangle_{2x}\oplus \langle e_4\rangle_{3x}\oplus\cdots \oplus \langle e_{n-1}\rangle_{(n-2)x}\oplus \langle e_n\rangle_{(n-1)x},$$
which is clearly equivalent to the  $\mathbb{Z}$-grading in Lemma \ref{gradings2}-(3).

		\item We have $1 \leq i\leq n-1.$ As above we get that $e_{j}\in A_{(j-1)x}$ for $1\geq j \geq n$. Hence, in case $i=1$ (and so $x=0$) we get that $A$ has the trivial grading. Suppose then
 $2 \leq i\leq n-1.$ Then
		$$\hbox{$e_1,e_2\in A_x,\ e_3\in A_{2x},\dots, e_{i}\in A_{(i-1)x}$
		with $|\{x,2x,\dots,(i-1)x\}|=i-1,$}$$
and also  $$e_i = e_{i-1} e_1\in A_{i x}=A_0$$
because ${\rm{ord}}(x)=i$.

		Finally, let  $j\geq i+1$ be. Since   $j=ir+s$ with $r,s \in {\mathbb N}$ and $0\leq s<i,$ we get as consequence of ${\rm{ord}}(x)=i$ that
$$e_j \in A_{jx}=-A_{s x}.$$ From here, if we write   $n=im+p$ with $m,p \in {\mathbb N}$ and $0\leq p<i,$ we have the grading
\begin{longtable}{lllllllllll}
		$A$ &$=$& 		&&    $\langle e_{i+1},$ &$e_{2i+1},$  &$\dots,$&$e_{1+(m-1)i},$& $e_{1+mi}\rangle_{\overline{0}}$\\
		&&$\oplus$ & $\langle e_1, e_2,$ &$e_{2+i},$&$e_{2+2i},$&$\dots,$&$e_{2+(m-1)i},$ & $e_{2+mi}\rangle_{\overline{1}}$ \\

		&&  $\oplus$ & $\langle e_3,$ &$e_{3+i},$ &$e_{3+2i},$ &$\dots,$ &$e_{3+(m-1)i},$ & $e_{3+mi}\rangle_{\overline{2}}$& \\
		&&  &$\ldots$ & \\
		&&  $\oplus$ & $\langle e_{p},$ &$e_{p+i},$ &$e_{p+2i},$ &$\dots,$ &$e_{p+(m-1)i},$ &  $e_{p+mi}\rangle_{\overline{{p-1}}}$ \\
		&&  $\oplus$ & $\langle e_{p+1},$ &$e_{p+1+i},$ &$e_{p+1+2i},$ &$\dots,$ &$e_{p+1+(m-1)i}\rangle_{\overline{p}}$ \\
		&&  & $\ldots$ \\
		&&  $\oplus$ & $\langle e_{i},$ &$e_{2i},$ &$e_{3i},$ &$\dots,$& $e_{mi}\rangle_{\overline{i-1}}.$\\
		\end{longtable}

This grading is equivalent to the 	$\mathbb{Z}$-grading in Lemma \ref{gradings2}-(3) when $(i,n)=(2,3)$ or to the	$\mathbb{Z}_i$-grading in Lemma \ref{gradings2}-(8) when $(i,n) \neq (2,3)$.
	\end{itemize}

%\bigskip
	
	In the second case $x \neq y.$ That is $e_1\in A_x, e_2 \in A_y$ with $x \neq y$.  We distinguish three possibilities.

	\begin{itemize}

\item First $i=1$. Then $x=0$ and so $$e_j= (\ldots ((e_2\underbrace{e_1)e_1)\dots)e_1}_{j-2}\in A_{y}$$
for any  $3 \leq j\leq n$. Hence
$$A=\langle e_1\rangle_{0} \oplus \langle e_2,e_3,\dots,e_n\rangle_{y},$$
which is equivalent to the $\mathbb{Z}_2$-grading in Lemma \ref{gradings2}-(5).

		\item Second, $i \geq n-1.$ We have  as above that $$e_j\in A_{y+(j-2)x}$$
 for any  $3 \leq j\leq n$.  Observe that the fact ${\rm ord}(x)\geq n-1$ gives us that $y+px \neq y+qx$ and $ y+px \neq y$ for any $1\leq p, q \leq n-2$ with $p \neq q$. However, it is possible that $y+px =x$ for some $1\leq p \leq n-2$. That is, $y=-kx$ for some $0\leq k \leq n-3$. Form here, we have that to distinguish two cases so as to obtain:

If $y \neq -kx$ for any $k \in \{0,1,...,n-3\}$ we get the grading
$$A=\langle e_1\rangle_x\oplus \langle e_{2}\rangle_{y}\oplus \langle e_{3}\rangle_{y+x}\oplus \langle e_4\rangle_{y+2x}\oplus\cdots \oplus \langle e_{n-1}\rangle_{y+(n-3)x}\oplus \langle e_n\rangle_{y+(n-2)x},$$
which is equivalent to the $\mathbb{Z}$-grading in Lemma \ref{gradings2}-(2).

If $y = -kx$ for some $k \in \{0,1,...,n-3\}$ then we have
$$\begin{array}{l}
A=\langle e_{2}\rangle_{-kx} \oplus \langle e_{3}\rangle_{(-k+1)x} \oplus \cdots \oplus \langle e_{k+1}\rangle_{-x} \oplus \langle e_{k+2}\rangle_{{{0}}} \oplus \langle e_1,e_{k+3}\rangle_{x} \oplus\\ \qquad \oplus \langle e_{k+4}\rangle_{{{2x}}} \oplus  \cdots  \oplus \langle e_{n}\rangle_{{{(n-k-2)x}}},
\end{array}$$
which is equivalent to the  $\mathbb{Z}$-grading in Lemma \ref{gradings2}-(4).

		\item Third, $2 \leq i\leq n-2.$ Then,
	$$\hbox{	$e_1 \in A_x, e_2 \in A_y, \ e_3\in A_{y+x},\dots, e_{i+1}\in A_{y+(i-1)x}$
		with $|\{x,2x,\dots,(i-1)x\}|=i-1.$}$$

We also have $$e_{i+2}  \in A_{y+ix}=A_{y}.$$
		Take now some  $j\geq i+3$ and express   $j-2=ir+s$ with $r,s \in {\mathbb N}$ and $0\leq s <i,$ then
$$e_j \in A_{y+(j-2)x}=A_{s}.$$

Now if we express $n=im+p$ with $m,p \in {\mathbb N}$ and $0\leq p <i,$ we can  distinguish as above two possibilities.

In the first one $y \neq -kx$ for any $k \in \{0,1,...,n-3\}$. Then we get the grading
\begin{longtable}{lllllllllll}
		$A$ &=&  & $\langle e_1  \rangle_{x}$\\
		&&  $\oplus$&$  \langle  e_2,$ & $e_{2+i},$ &$e_{2+2i},$&$\ldots,$&$e_{2+(m-1)i},$&$e_{2+mi}\rangle_{y}$ \\
		&&  $\oplus$&$  \langle   e_3,$ &$e_{3+i},$ &$e_{3+2i},$ &$\ldots,$ &$e_{3+(m-1)i},$ &$e_{3+mi}\rangle_{y+x}$ \\
		&&  &$\ldots$ & \\
		&&  $\oplus $&$ \langle  e_{p},$ &$e_{p+i},$ &$e_{p+2i},$ &$\ldots,$ &$e_{p+(m-1)i},$ &$e_{p+mi}\rangle_{y+(p-2)x}$ \\
		&&  $\oplus$&$  \langle  e_{p+1},$ &$e_{p+1+i},$ &$e_{p+1+2i},$ &$\ldots,$ &$e_{p+1+(m-1)i}\rangle_{y+(p-1)x}$ \\
		&&  & $\ldots$ \\
		&&  $\oplus$&$  \langle  e_{i+1},$ &$e_{2i+1},$ &$e_{3i+1},$ &$\ldots,$& $e_{i+1+(m-1)i}\rangle_{y+(i-1)x},$
		\end{longtable}
\noindent which is equivalent to the  $\mathbb{Z} \times \mathbb{Z}_2$-grading in Lemma \ref{gradings2}-(7) when $i=2$ or to $\mathbb{Z} \times \mathbb{Z}_i$-grading in Lemma \ref{gradings2}-(10) when $i \geq 3$.

In the second one $y = -kx$ for some $k \in \{0,1,...,n-3\}$.  Then we have  the grading
\begin{longtable}{lllllllllll}
		$A$   &$=$&             &                     &$\langle  e_{1},$ & $e_{k+3},$& $e_{k+3+i},$ & $e_{k+3+2i},$  & $\ldots  \rangle_{x}$ \\
		    &&          &$\oplus$&    $\langle   e_2,$ & \ $e_{k+4},$            & $e_{k+4+i},$ & $e_{k+4+2i},$  & $\ldots  \rangle_{2x}$ \\
		    &&         & &$\ldots$& \\
&&             &$\oplus$                      &$\langle  e_{k+1}, $& $e_{k+1+i},$ & $e_{k+1+2i},$  & $e_{k+1+3i},$ & $\ldots  \rangle_{(i-1)x}$ \\

		    &&             &$\oplus$                      &$\langle  e_{k+2}, $& $e_{k+2+i},$ & $e_{k+2+2i},$  & $e_{k+2+3i},$ & $\ldots  \rangle_{0}.$
\end{longtable}
\noindent which is  equivalent  either to the $\mathbb{Z}_2$-grading in \ref{gradings2}-(6) when $k=0$ (that is, $y=0$) and $i=2$ or to the $\mathbb{Z}_i$-grading in \ref{gradings2}-(9) when $k=0$  and $i\geq 3$,  or to the  $\mathbb{Z} \times \mathbb{Z}_2$-grading in Lemma \ref{gradings2}-(7) when $k \neq 0$ (that is, $y \neq 0$) and $i=2$ or
 to the  $\mathbb{Z}_i$-grading in \ref{gradings2}-(11) when $k \neq 0$ and $i \geq 3$.

	\end{itemize}

	\end{proof}

\begin{teo}
Any group grading of a one-parametric  filiform Leibniz algebra  with  $\theta=0$, and dimension $n \geq 3$,   is equivalent to only one in the list of Lemma \ref{gradings2}.
\end{teo}

\begin{proof}
 Recall  that any grading of   $A$  is induced by an abelian subgroup $\mathcal{G}$ of diagonalizable automorphisms in $\mathcal{N}(\mathcal{T})=\mathcal{T}$, being the homogeneous components of the grading the common eigenspaces of the elements in $\mathcal{G}.$

Since for any $f \in \mathcal{G}\subset \mathcal{T}$, Equation (\ref{toruss}) gives us that  $e_1$ and $e_2$ are eigenvectors of $f$, then $e_1$ and $e_2$ are homogeneous vectors in any grading of $A$. From here Lemma \ref{JC} completes the proof.
\end{proof}

\subsection{Case in which $\theta\neq 0$}

Taking into account Equation (\ref{June1}), we have that in  case ${\rm dim}(F_1) \geq 4$ then

$${{\rm Aut}}(F_1)=
\left\{\left(\begin{array}{lllll}
\epsilon&0&0&\dots&0\\
0&b&0&\dots&0\\
0&b_3&\epsilon b&\dots&0\\
0&b_4&\epsilon b_3&\dots&0\\
\vdots&\vdots&\vdots&\ddots&\vdots\\
a_n&b_n&\epsilon b_{n-1}&\dots&\epsilon^{n-2} b

\end{array}\right):
a_n, b_3,...,b_n \in \mathbb{C},  b\in \mathbb{C}^{*}; \epsilon^{n-3}=1  \right\}.$$

 From here  a maximal torus is:
$$\mathcal{T}=\left\{t_{1,b}:=\left(\begin{array}{ccccc}
1&0&0&\dots&0\\
0&b&0&\dots&0\\
0&0&b&\dots&0\\
\vdots&\vdots&\vdots&\ddots&\vdots\\
0&0&0&\dots&b

\end{array}\right):b\in \mathbb{C}^{*}\right\} \cong \mathbb{C}^{*}.$$

Hence, the normalizer of $\mathcal{T}$ in ${{\rm Aut}}(F_1)$ is

$$\mathcal{N}(\mathcal{T})=\left\{\left(\begin{array}{lllll}
\epsilon&0&0&\dots&0\\
0&b&0&\dots&0\\
0&b_3&\epsilon b&\dots&0\\
0&b_4&\epsilon b_3&\dots&0\\
\vdots&\vdots&\vdots&\ddots&\vdots\\
0&b_n&\epsilon b_{n-1}&\dots&\epsilon^{n-2}b

\end{array}\right):b_3,...,b_n \in \mathbb{C},  b\in \mathbb{C}^{*}; \epsilon^{n-3}=1  \right\}.$$

Observe that $\mathcal{T} \subsetneqq \mathcal{N}(\mathcal{T})$ and so any grading of $F_1$ is not necessarily a toral one.

\medskip

Let us denote by $A$ any (non-Lie) one-parametric filiform Leibniz  with $\theta\neq 0$ and with ${\rm dim}(A) \geq 3$
(the cases ${\rm dim}(A) =1,2$ give algebras with zero product).

\medskip

 Recall now that any grading of   $A$  is induced by an abelian subgroup $\mathcal{G}$ of diagonalizable automorphisms in $\mathcal{N}(\mathcal{T})$, being the homogeneous components of the grading the common eigenspaces of the elements in $\mathcal{G}.$

Since for any $f \in \mathcal{G}$ (and following from the multiplication table of the algebra), $e_1$ and $e_n$ are eigenvectors of $f$, then $e_1$ and $e_n$ are homogeneous vectors in any grading of $A$. That is,
\begin{equation}\label{queen1}
\hbox{$e_1 \in A_g$ and $e_n \in A_h$}
\end{equation}

 for some $g,h \in {G}$, where $G$ denotes the abelian group in the grading of $A$:
\begin{equation}\label{June2}
A=\bigoplus\limits_{g \in G} A_g.
\end{equation}

\medskip

If $A$ is three-dimensional,  we know by Equation (\ref{queen1}) that $e_1 \in A_g$ and $e_3 \in A_h$. We are going to distinguish two possibilities:

\medskip

First, $g=h$. In this case either $A=A_g$ when $e_2 \in A_g$ or $A=A_g \oplus A_t$ with $t \neq g$ when $e_2 \notin A_g$. In the first case we get the trivial grading $A=A_0$. In the second case we can write $e_2=v_g+ v_t$ with $v_g \in A_g$ and $0\neq v_t \in A_t$. If $v_g=0$ then $e_2 \in A_t$ and  we have the ${\mathbb Z}_2$-grading
 \begin{equation}\label{queen4}
 A= \langle  e_2 \rangle_{\overline{0}} \oplus \langle e_1, e_3 \rangle_{\overline{1}}.
 \end{equation}
If $v_g \neq 0$  then
the fact $e_1e_2= \theta e_3$ and the grading of $A$ give us $2g=g$ and $g+t=g$. From here $g=0$ and so $t=0$. Hence $g=t$, a contradiction, and so this case does not happen.

\medskip

Second, $g \neq h$. In this case either $A=A_g \oplus A_h$ when $e_2 \in A_g\oplus A_h$ or $A=A_g \oplus A_h \oplus A_t$ with $t \notin \{g,h\}$  when $e_2 \notin A_g\oplus A_h$.

 Consider the case in which $e_2 \in A_g\oplus A_h$ and write $e_2=v_g+v_h$ with $v_g \in A_g$ and  $v_h \in A_h$.  If $v_g=0$ we get $e_2=e_h$ and the  the ${\mathbb Z}_2$-grading
\begin{equation}\label{queen5}
A= \langle  e_1 \rangle_{\overline{0}} \oplus \langle e_2, e_3 \rangle_{\overline{1}}.
\end{equation}
If $v_h=0$ then $e_2=v_h$ and we obtain the the ${\mathbb Z}_2$-grading

 \begin{equation}\label{queen6}
 A= \langle  e_3 \rangle_{\overline{0}} \oplus \langle e_1, e_2 \rangle_{\overline{1}}.
 \end{equation}
 Finally if $v_g \neq 0$ and $v_h \neq 0$, the fact $e_1e_2= \theta e_3$ and the grading of $A$ give us $2g=h$ and $g+h=h$. From here $g=h=0$. This is a contradiction  and so this case does not happen.

 Consider now the case in which $e_2 \notin A_g\oplus A_h$. Then we can write $e_2 = v_g+v_h+v_t$ with $v_i \in A_i$ for $i \in \{g,h,t\}$ and $v_t \neq 0$. If $v_g=v_h=0$ then $e_2 =v_t$ and we get the ${\mathbb Z}$-grading
  \begin{equation}\label{queen7}
 A= \langle  e_1 \rangle_{1} \oplus \langle e_2 \rangle_{2}\oplus \langle e_2 \rangle_{3}.
 \end{equation}
 If $v_g \neq 0$ we get as above that $2g=h$ and that $g+t=h$. From here $t=g$ a contradiction. If $v_h \neq 0$ we have in a similar way  $g+h=h$ and $g+t=h$. From here $h=t$. A contradiction. Hence these last cases do not occur.

Since the grading (\ref{queen7}) has three homogeneous components while the remaining non-trivial gradings have only two homogeneous components  we have that the grading (\ref{queen7})is not equivalent to any grading in (\ref{queen4})-(\ref{queen6}). The grading (\ref{queen4}) is not equivalent to grading (\ref{queen6}). Indeed, in the opposite case  we would have an automorphism $\phi$ of $A$ satisfying
$\phi(\langle e_2\rangle)= \langle e_3 \rangle$ and $\phi(\langle e_1, e_3 \rangle)=\langle e_2, e_3 \rangle$. But in this case $0 \neq  \phi(e_3)=\phi(e_2) \phi(e_1)=0$, a contradiction. In a similar way, gradings (\ref{queen5}) and (\ref{queen6}) are not equivalent. Finally, an analogous argument gives us that gradings (\ref{queen4}) and (\ref{queen5}) are equivalent if and only if $\theta \in \{\pm 1\}$.

%\bigskip

Let us summarize our results in the following statement.

\begin{proposition}
Any abelian group grading of a one-parametric  filiform  Leibniz algebra $A$ (with a nonzero parameter) of dimension $3$ is equivalent to the   trivial grading or to one of the  gradings (\ref{queen4}), (\ref{queen5}), (\ref{queen6}) or (\ref{queen7}) if $\theta \notin \{\pm1\}$; or to the trivial grading or to one of the  gradings (\ref{queen4}),  (\ref{queen6}) or (\ref{queen7}) if $\theta \in \{\pm1\}.$

\end{proposition}

%\medskip

Below we will consider the case  of the algebra $A$ for   $n \geq 4$.

Now, since $e_1e_2=\theta e_n$ with $\theta \neq 0$, we can write by Equation (\ref{June2}) that $$e_2=v_{g_1}+v_{g_2}+\cdots+v_{g_m}$$ with any $v_{g_i}\neq 0$ and
$g_i\neq g_j$ when $i \neq j$.
From here $e_1(v_{g_1}+v_{g_2}+\cdots+v_{g_m})=\theta e_n$ and,
by distinguish removing the zero and the non-zero products, we can write
$$\sum\limits_{j\in J} e_1v_{g_j}+\sum\limits_{k\in K} e_1v_{g_k}=\theta e_n,  \ J \cap K= \emptyset,$$ with $e_1v_{g_j}=0$ for any $j \in J$ and $e_1v_{g_k} \neq 0$ for any $k \in  K$. Since $e_1 \in A_g$, then the fact $\sum\limits_{k\in K} e_1v_{g_k}=\theta e_n$ implies that by writing for any $k \in K$,
$0 \neq e_1 v_{g_k}=v_{g_k+g} \in A_{g_k + g}$, we have that $0\neq \sum\limits_{k\in K} v_{g_k+g}=\theta e_n \in A_h.$ Hence by the grading, we get that $g_k+g =h$ for any $k \in K$.
That is, $g_k=h-g$ for any $k \in K$. From here we have that the cardinal of $K$ is necessarily 1 and we write
$$e_2=\sum\limits_{j\in J} v_{g_j} + v_{h-g}$$ where  any $v_{g_j} \in A_{g_j}$  with $e_1v_{g_j}=0$ and $0\neq v_{h-g} \in A_{h-g}$ with $e_1v_{h-g}\neq 0$.

Now, since $e_2e_1=e_3$, we have $\sum\limits_{j\in J} v_{g_j}e_1 + v_{h-g}e_1=e_3$. In case $v_{h-g}e_1=0$, then by the product in $A$ we get
$0 \neq v_{h-g}= \alpha e_1+ \beta e_n$, $\alpha, \beta \in {\mathbb C}$. But  $v_{h-g}=e_2-\sum\limits_{j\in J} v_{g_j}$ and so $e_1v_{h-g}=e_1e_2=e_3\neq 0$. However in this case
$e_1v_{h-g}=e_1( \alpha e_1+ \beta e_n)=0$ which is a contradiction (recall $n \geq 4$). Hence $v_{h-g}e_1\neq 0$ and we can write
$0 \neq v_{h-g}e_1:=v_h \in A_h.$ We get $$e_3= \sum\limits_{j\in J} v_{g_j+g} + v_h$$
with $v_h \neq 0$.

\medskip

We have that $v_he_1\neq 0$. Indeed, if $v_he_1= 0$ then $v_h=\alpha e_1 + \beta e_n$ with $\alpha \neq 0$ or $\beta \neq 0$. But
$0 \neq v_h=v_{h-g}e_1$ and so (by the product in $A$),
$\alpha =0$, $\beta \neq 0$ and $v_{h-g}= \tau e_1+ \beta e_{n-1} + \gamma e_n$. Recall from the above that $e_1v_{h-g} \neq 0$ but $e_1( \tau e_1+ \beta e_{n-1} + \gamma e_n)=0$, a contradiction.
From here $v_he_1 \neq 0$ and by denoting $0 \neq v_he_1:=v_{h+g} \in A_{h+g}$ we have that $e_4=e_3e_1=(\sum\limits_{j \in J} v_{g_j+g}+ v_h) e_1=
  \sum\limits_{j \in J} v_{g_j+2g}+ v_{h+g}$ with $v_{h+g} \neq 0$. That is, we can assert  $$e_4=\sum\limits_{j \in J} v_{g_j+2g}+ v_{h+g}$$ with $v_{h+g} \neq 0$.

 By arguing in this way get  for any $k \in \{2,...,n-1\}$ that
 $$e_k=\sum\limits_{j \in J} v_{g_j+(k-2)g}+ v_{h+(k-3)g}$$ with $v_{h+(k-3)g} \neq 0$.

 Finally, for $k=n$ we have that the fact $e_{n-1}e_1=e_n$ gives us
 $e_n=e_k=\sum\limits_{j \in J} v_{g_j+(n-2)g}+ v_{h+(n-3)g}$ with $v_{h+(n-3)g} \neq 0$ and $e_n \in A_h$. From here, we get by the grading that
   $h+(n-3)g=h$ and so (recall $n \geq 4)$,
  \begin{equation}\label{air}
 (n-3)g=0.
 \end{equation}

 We have by the above that $(n-2)g=g$. So  in case some $ v_{g_j+(n-2)g}\neq 0$ we have  $ v_{g_j+(n-2)g} \in A_{g_j+g}=A_h$ and so $g_j=h-g$. Summarizing we have (take into account that the fact $(n-3)g=0$ implies $(n-4)g=-g$).
 \begin{equation}\label{serkan}
 e_1 \in A_g, \hspace{0,25cm}e_2 \in A_{h-g},\hspace{0,25cm}e_3 \in A_{h},\hspace{0,25cm}e_4 \in A_{h+g}, \cdots , e_k \in A_{h+(k-3)g}, \cdots, e_{n-2} \in A_{h+(n-5)g},
 \end{equation}
 $$e_{n-1} \in A_{h+(n-4)g}=A_{h-g},\hspace{0,25cm}e_n \in A_{h}.$$

%\bigskip
Let us distinguish two cases:

%\medskip

First, $g=0$. Then, by Equation (\ref{serkan}),  we have either the trivial grading $$A=A_0$$ in case $h=0$ or the ${\mathbb Z}_2$-grading
\begin{equation}\label{inicio}
 A=\langle e_1 \rangle_{\overline{0}}\oplus
\langle e_2,e_3,...,e_n \rangle_{\overline{1}}
\end{equation}

 when $h \neq 0$.

%\medskip
Second, $g \neq 0$. Observe that ${\rm dim}(A) \neq 4$ since in this case Equation (\ref{air}) would give us $g=0$.

 Let us denote by $2\leq t \leq n-3$ the order of $g$. By Equation (\ref{air}),  $t$ divides $n-3$ and so we can write
\begin{equation}\label{selin}
 {n-3}=rt
 \end{equation}

 with $t,r \in {\mathbb N}$ being  $2\leq t\leq n-3$  and   $1 \leq r \leq n-4$.

%\smallskip

 If furthermore  $h=g$, Equation (\ref{serkan}) gives us the following grading,
 which will be denote as
 \begin{equation}
     \label{eda}
 A_g  =\langle e_1 \rangle \oplus \langle e_3,e_{t+3},e_{2t+3},...,e_{(r-1)t+3} \rangle \oplus \langle e_n \rangle
 \end{equation}
\begin{longtable}{rcl} $A_{2g}$& $=$ & $ \langle e_4,e_{t+4},e_{2t+4},...,e_{(r-1)t+4} \rangle$\\
 $A_{3g}$ & $=$ &$ \langle e_5,e_{t+5},e_{2t+5},...,e_{(r-1)t+5}\rangle $\\
  & $\vdots$&\\
$A_{(t-1)g}$ & $=$ & $ \langle e_{t+1},e_{2t+1},e_{3t+1}, ...,e_{rt+1} \rangle $\\
$A_{0}$ & $=$ & $\langle e_{t+2},e_{2t+2},e_{3t+2},...,e_{rt+2} \rangle \oplus \langle e_2 \rangle$
\end{longtable}
%\medskip

From here, for any divisor $t \neq 1$ of $n-3$ we have a ${\mathbb Z}_t$-grading of $A$  by taking ${\overline i}:=ig$, $i \in \{0,...,t-1\}$ in Equation (\ref{eda}) and where
$rt=n-3$.

%\medskip

We note that $pg \neq qg$ for $p,q \in \{0,...,t-1\}$, $p \neq q$,  in Equation (\ref{eda}) since in the opposite case $(p-q)g=0$ with $p-q \leq t-1$, a contradiction with the order $t$ of $g$.

%\medskip
We also note that for $t=2$, Equation (\ref{eda}) means the ${\mathbb Z}_2$-grading of $A$:
\begin{longtable}{rcl}
$A_{\overline{0}}$ & $=$ & $\langle e_{4},e_{6},e_{8},...,e_{2(r+1)} \rangle \oplus \langle e_2 \rangle$\\
$A_{\overline 1}$ & $=$ & $\langle e_1 \rangle \oplus \langle e_3,e_{5},e_{7},...,e_{2r+1} \rangle \oplus \langle e_n \rangle.$
\end{longtable}

%\medskip

It remains to study the case in which $h \neq g$ (and also $g \neq 0$). In this case Equation (\ref{serkan})  gives us
\begin{equation}\label{casi}
 A_g=\langle e_1 \rangle
 \end{equation}
\begin{longtable}{rcl}
$A_{h}$ & $=$&$\langle e_3,e_{t+3}, e_{2t+3},...,e_{(r-1)t+3} \rangle \oplus \langle e_n \rangle $\\
$A_{h+g}$ & $=$&$ \langle e_4,e_{t+4}, e_{2t+4},...,e_{(r-1)t+4} \rangle$\\
$A_{h+2g}$ & $=$ & $\langle e_5,e_{t+5}, e_{2t+5},...,e_{(r-1)t+5} \rangle$\\
  &$\vdots$&\\
$A_{h+(t-2)g}$& $=$& $ \langle e_{t+1},e_{2t+1}, e_{3t+1},...,e_{rt+1} \rangle$\\
$A_{h+(t-1)g}$ & $=$ & $\langle e_2 \rangle \oplus \langle e_{t+2},e_{2t+2}, e_{3t+2},...,e_{rt+2} \rangle. $
\end{longtable}

We have as above that $h+pg \neq h+qg$ for $p,q \in \{0,...,t-1\}$. However it is possible that
$g=h+pg$ for some $p \in \{1,...,t-1\}$. Hence we consider two possibilities:

\bigskip

First, $g\neq h+pg$ for any $p \in \{1,...,t-1\}$. Then we get  that  for any divisor $t \neq 1$ of $n-3$ we have a ${\mathbb Z}_t \times {\mathbb Z}$-grading of $A$ given by Equation (\ref{casi}) where $({\overline 1},0):=g$,
   $({\overline i},1):=h+ig$   for   $i \in \{0,...,t-1\}$, and where
$rt={n-3}$.

We note that for $t=2$, Equation (\ref{casi}) means the ${\mathbb Z}_2 \times {\mathbb Z}$-grading of $A$:

\begin{longtable}{rcl}
$A_{({\overline 1},0)}$ & $=$ & $\langle e_1 \rangle$\\
$A_{({\overline 0},1)}$ & $=$ & $\langle e_3,e_{5}, e_{7},...,e_{2r+1} \rangle \oplus \langle e_n \rangle $\\
$A_{({\overline 1},1)}$ & $=$ & $\langle e_2 \rangle \oplus \langle e_{4},e_{6}, e_{8},...,e_{2(r+1)} \rangle. $
\end{longtable}

\bigskip

Second, $g=  h+pg$ for some $p \in \{1,...,t-1\}$. This fact is equivalent to $h=-\tau g$ for some $\tau \in \{0,...,t-2\}$. In this case $$A_g=A_{h+ (\tau +1) g}$$ and, by looking at Equation (\ref{casi}),  we obtain the grading

\begin{equation}\label{ultima}
 A_g=\langle e_1 \rangle \oplus \langle e_{\tau+4},e_{t+\tau+4},...,e_{(r-1)t+\tau+4}\rangle
 \end{equation}
\begin{longtable}{rcl}
$A_{-\tau g}$ & $=$&$\langle e_3,e_{t+3}, e_{2t+3},...,e_{(r-1)t+3} \rangle \oplus \langle e_n \rangle $\\
$A_{(-\tau +1) g}$ & $=$&$ \langle e_4,e_{t+4}, e_{2t+4},...,e_{(r-1)t+4} \rangle$\\
$A_{(-\tau +2) g}$ & $=$ & $\langle e_5,e_{t+5}, e_{2t+5},...,e_{(r-1)t+5} \rangle$\\
  &$\vdots$&\\
  $A_{0}$ & $=$ & $\langle e_{\tau +3}, e_{t+\tau +3},...,e_{(r-1)t+\tau +3} \rangle$\\
  &$\vdots$&\\
$A_{(t-\tau-2)g}$& $=$& $ \langle e_{t+1},e_{2t+1}, e_{3t+1},...,e_{rt+1} \rangle$\\
$A_{(t-\tau-1)g}$ & $=$ & $\langle e_2 \rangle \oplus \langle e_{t+2},e_{2t+2}, e_{3t+2},...,e_{rt+2} \rangle. $
\end{longtable}

That is,  for any divisor $t \neq 1$ of $n-3$ and for any $\tau \in \{0,...,t-2\}$ we have a ${\mathbb Z_t}$-grading of $A$ given by Equation (\ref{ultima}) where  $A_{{\bar i}}:=A_{ig}$ for any ${\bar i} \in {\mathbb Z_t}$ and where $rt={n-3}$.

\begin{teo}\label{teo10}
Let $A$ be a one-parametric  filiform Leibniz algebra $A$ of dimension $n$.
\begin{itemize}
\item[(i)] If $n=4$ then any abelian group grading of $A$ is equivalent   either to the trivial grading or to the  grading (\ref{inicio}).
\item[(ii)]  If  $n \geq 5$  then any abelian group grading of $A$ is equivalent    either to the trivial grading or to the  grading (\ref{inicio}) or to one of the gradings in   (\ref{eda})  or  to one of the gradings in (\ref{casi})  or to one of the gradins in (\ref{ultima}).

\end{itemize}
\end{teo}

\begin{proof}

We only have to show that,  for $n \geq 4$, the grading (\ref{inicio}), any grading in  (\ref{eda}),  any grading in  (\ref{casi}) and any grading in  (\ref{ultima}) are not equivalent.  To do that, observe that the grading
(\ref{inicio}) has one homogeneous component with dimension 1 and one homogeneous component with dimension $n-1$.

Fixed now a divisor $t\geq 2$ of $n-3$ and by denoting $rt={n-3}$, we have a grading given by (\ref{eda}) which has
   one homogeneous component with dimension $r+1$, one homogeneous components with dimension $r+2$ and $t-2$ homogeneous components with dimension $r$;  a grading  (\ref{casi}) which has
  one homogeneous component with dimension $1$, two  homogeneous components with dimension $r+1$ and $t-2$ homogeneous components with dimension $r$; and  a grading  (\ref{ultima}) with either three  homogeneous components with dimension $r+1$ and $t-3$ homogeneous components with dimension $r$; or with one  homogeneous component with dimension $r+2$, one  homogeneous component with dimension $r+1$ and $t-2$ homogeneous components with dimension $r$

  From here, the grading (\ref{inicio}) is  equivalent neither  to any grading in (\ref{eda}) nor  any grading in  (\ref{casi}) and nor  any grading in  (\ref{ultima}) (they have a different number of nonzero homogeneous components). Also,  two different gradings in (\ref{eda}), two different gradings in (\ref{casi}) and two different gradings in (\ref{ultima}) are not equivalent because they have nonzero homogeneous components with different dimensions.

  Now, let us  show that any grading in (\ref{eda}) is not equivalent to any grading in (\ref{casi}). To do that, suppose there exists $t,t' \geq 2$ two divisors of $n-3$ such that the ${\mathbb Z}_t$-grading of $A$:
 \begin{equation}\label{omar1}
 A_{\overline{0}}= \langle e_{t+2},e_{2t+2},e_{3t+2},...,e_{rt+2} \rangle \oplus \langle e_2 \rangle
 \end{equation}
\begin{longtable}{rcl}
$A_{\overline{1}}$ & $=$ & $\langle e_1 \rangle \oplus \langle e_3,e_{t+3},e_{2t+3},...,e_{(r-1)t+3} \rangle \oplus \langle e_n \rangle$\\
$A_{\overline{2}}$ & $=$ & $\langle e_4,e_{t+4},e_{2t+4},...,e_{(r-1)t+4} \rangle$\\
$A_{\overline{3}}$ & $=$ & $\langle e_5,e_{t+5},e_{2t+5},...,e_{(r-1)t+5}\rangle $\\
  &$\vdots$& \\
$A_{\overline{t-1}}$ & $=$ & $\langle e_{t+1},e_{2t+1},e_{3t+1}, ...,e_{rt+1} \rangle $
\end{longtable}
and the ${\mathbb Z}_{t'} \times {\mathbb Z}$-grading of $A$
 \begin{equation}\label{omar2}
 A_{({\overline 1},0)}=\langle e_1 \rangle
 \end{equation}
\begin{longtable}{rcl}
$A_{({\overline 0},1)}$ & $=$ & $\langle e_3,e_{t'+3}, e_{2t'+3},...,e_{(r'-1)t'+3} \rangle \oplus \langle e_n \rangle $\\
$A_{({\overline 1},1)}$ & $=$ & $\langle e_4,e_{t'+4}, e_{2t'+4},...,e_{(r'-1)t'+4} \rangle$\\
$A_{({\overline 2},1)}$ & $=$ & $\langle e_5,e_{t'+5}, e_{2t'+5},...,e_{(r'-1)t'+5} \rangle$\\
  &$\vdots$& \\
$A_{(\overline{t'-2},1)}$ & $=$ & $ \langle e_{t'+1},e_{2t'+1}, e_{3t'+1},...,e_{r't'+1} \rangle$\\
$A_{(\overline{t'-1},1)}$ & $=$ & $ \langle e_2 \rangle \oplus \langle e_{t'+2},e_{2t'+2}, e_{3t'+2},...,e_{r't'+2} \rangle $
\end{longtable}
  are equivalent, where $rt={n-3}$ and $r't'={n-3}$.
Since both gradings must have the same number of nonzero homogeneous components, then  necessarily
   \begin{equation}\label{omar4}
  t=t'+1.
  \end{equation}

  \medskip

  As the grading (\ref{omar2}) has an homogeneous  component with dimension 1, then necessarily $r=1$ and so $t=n-3$ in the grading (\ref{omar1}).   Then we have  by Equation (\ref{omar4}) that $t'=n-4$. Since $t'$ divides $n-3$ then necessarily $n-3=2$ and $t'=1$, but this contradicts the fact that $t \geq 2$.  We conclude that
  the gradings  in (\ref{omar1}) and in (\ref{omar2}) are not equivalent.

  \medskip

  Similar arguments allow us to verify that  any grading in (\ref{eda}) is not equivalent to any grading in (\ref{ultima}); and that   any grading in (\ref{casi}) is not equivalent to any grading in (\ref{ultima}).
      % in this case the grading
 % (\ref{omar1}) has $n-5$ homogeneous components with dimension 1,  one homogeneous component with dimension $2$ and one homogeneous component with dimension $3.$

%  From here, the grading (\ref{omar2}) must have an homogeneous component with dimension 3, but this is not possible by Equation (\ref{omar4}). We conclude that
%  the gradings  (\ref{omar1}) and (\ref{omar2}) are not equivalent.
\end{proof}

%%%%%%%%%%%%%%%%%%%%%%%%%%%%%%%%%%%%%%%%%%%%%%%%%%%%%%%%%%%%%%%%%%%%%%%%%%%%%%%%%%%%%%%%%%%%%%%%%%%%%%%%%%%%%%%%%%%%%%%%%%%%%%%%%%%%%%%%%%%%%%%%%%%%%%%%%%%%%%%5

\color{black}

%\newpage

\end{document}